\documentclass[11 pt]{amsart}

\usepackage{fullpage}
\usepackage{amssymb,amsmath,amsfonts, amscd,url}
\usepackage{array}
\usepackage{enumitem}
\usepackage{graphicx}
\usepackage[hang]{subfigure}

\usepackage{color}

\newtheorem{theorem}{Theorem}

\newtheorem{lemma}[theorem]{Lemma}

\newtheorem{thm}[theorem]{Theorem}

\newcommand{\e}{{\epsilon}}

\newcommand{\cS}{\mathcal{S}}

\begin{document}

\title{Twins in graphs}
\author{Maria Axenovich}
\address{Department of Mathematics,  Karlsruhe Institute of Technology, Karlsruhe 76128, Germany}
\email{maria.aksenovich@kit.edu}
\thanks{This author's research  was partially supported by NSF grant DMS-0901008.}
\author{Ryan Martin}
\address{Department of Mathematics, Iowa State University, Ames, Iowa 50011}
\email{rymartin@iastate.edu}
\thanks{This author's research partially supported by NSF grant DMS-0901008, NSA grant H98230-13-1-0226 and by an Iowa State University Faculty Professional Development grant.}
\author{Torsten Ueckerdt}
\address{Department of Mathematics, Karlsruhe Institute of Technology, Karlsruhe 76128, Germany}
\email{torsten.ueckerdt@kit.edu}


\maketitle

\begin{abstract}
 A basic pigeonhole principle insures an existence of two objects of the same type
 if the number of objects is larger than the number of types. Can such a principle be extended to
 a more complex combinatorial structure?  Here, we address such a question for graphs.
We call two disjoint subsets $A, B$ of vertices   \emph{\textbf{twins}} if
 they have the same cardinality and induce subgraphs of the same size.
 Let $t(G)$ be the largest $k$ such that $G$ has twins on $k$ vertices each.
We provide the bounds on $t(G)$ in terms of the number of edges  and vertices using discrepancy results for  induced subgraphs.
In addition, we give conditions under which $t(G)= |V(G)|/2$ and show that if $G$ is a forest then $t(G) \geq |V(G)|/2 - 1$.
 \end{abstract}

\section{Introduction}
\label{sec:intro}

For a combinatorial structure $G$ and a set of parameters, we say that two disjoint substructures are \emph{twins} if those parameters coincide for each substructure.
This notion generalizes the pigeonhole principle and was investigated in case of sequences and graphs.  In the case of sequences it is known, see Axenovich, Person and Puzynina ~\cite{APP},  that any binary sequence
of length  $n$ contains two disjoint identical subsequences of length $n/2-o(n)$ each.
Twins are defined to be two disjoint  vertex subsets of the same size and the same multiset of pairwise distances were studied by   Albertson, Pach and Young~\cite{APY},  and by Axenovich and \"Ozkahya~\cite{AO}.
Edge-disjoint isomorphic subgraphs were also addressed in ~\cite{GR, LLS, AKS}.

Here, we concentrate on the following twin problem in graphs introduced by Caro and Yuster in~\cite{caroyuster}.
For a graph $G$, we call two disjoint subsets  of vertices \emph{twins} if  they have the same size and
induce subgraphs with the same number of edges.
Let $t(G)$ be the largest $k$ such that there are twins $A, B$ in $G$ with $|A|=|B|=k$.
Let $$t(n) = \min \{ t(G):  |V(G)|=n\}.$$ The best currently known bounds on $t(n)$ are given in the following theorem.

\begin{theorem}[Caro and Yuster~\cite{caroyuster}]\label{general-bounds}
There exists a positive constant $c$ such that $\sqrt{n} \leq t(n) \leq  n/2 - c\log\log n$.
\end{theorem}

 Ben-Eliezer and  Krivelevich,~\cite{BK},  proved that  $t(G(n,p))=\lfloor n/2\rfloor$ with high probability, where $G(n,p)$ is the Erd\H{o}s-Renyi random graph.
In addition, Caro and Yuster proved that in a sparse graph there are twins of size almost $n/2$:

\begin{thm}[Caro and Yuster~\cite{caroyuster}]\label{num-edges}
For every fixed $\alpha > 0$ and for every $\epsilon > 0$ there exists $N=N(\alpha, \epsilon)$
so that for all $n > N$, if $G$ is a graph on $n$ vertices and at most $n^{2-\alpha}$ edges
then $t(G)\geq (1-\epsilon)n/2$.
\end{thm}

In this paper, we do the following:
\begin{itemize}
   \item In Theorem~\ref{thm:size}, we improve the above result using a more general theorem on discrepancy for induced subgraphs.
   \item In Theorem~\ref{thm:exact}, we give several sufficient criteria under which the graph has \emph{perfect twins}; i.e., two twins spanning the whole graph.
   \item In Theorem~\ref{thm:forest}, we prove that all but at most two vertices of any forest could be split into two twins.
\end{itemize}

For disjoint vertex sets $A, B$   in a graph $G$, we denote by $e(A)$ the number of edges induced by $A$,
$e(A,B)$ the number of edges between $A$ and $B$.  {The} \emph{discrepancy} between $A$ and $B$ is ${\rm disc}(A,B) = |e(A) -e(B)|$.
Let $\lg n$  denote the logarithm of $n$ in base $2$.
The maximum and the minimum degree of a graph $G$ are denoted by $\Delta(G)$ and $\delta(G)$, respectively,
the set of vertices of degree $i$ is  $V_i$.  A pair of vertices is called \emph{consecutive} if the degrees of these vertices differ by exactly $1$ or $-1$.

\begin{thm}\label{thm:almosttwins}
   If $G$ is an $n$ vertex graph, $n\geq 16$, then there are vertex disjoint sets $A$ and $B$ such that
   \begin{itemize}
    \item ${\rm disc}(A,B) \leq 2\lg^2 n$ and $|A|=|B| \geq (n-2\lg n)/2$ or
    \item ${\rm disc}(A,B) \leq (\Delta(G)-\delta(G)+1)/2$ and $|A|=|B| = \lfloor n/2 \rfloor$.
   \end{itemize}
\end{thm}

\begin{theorem} \label{thm:size}
If $G$  is  a graph on  $n$ vertices and $e$ edges then
$t(G) \geq  \frac{n}{2} \left( 1- \frac{20 \sqrt{e} \lg n}{n}\right).$
\end{theorem}

This theorem implies, in particular, that any $n$-vertex graph on $o\left(n^2 / \lg^2n\right)$  edges  has twins of size $n/2 - o(n)$ each and
 that planar graphs have twins of size at least $n/2 - c \lg n$ each.

\begin{thm} \label{thm:exact}
 Let $G$ be  a graph on   $n$  vertices,  where $n$ is even. If one of the following conditions~\ref{enum:consecutive-degrees}-\ref{enum:disjoint-consecutive-pairs} holds then $t(G)=n/2$.
 \begin{enumerate}[label= \arabic*)]
 \item The degree sequence of $G$ forms a set of consecutive integers.\label{enum:consecutive-degrees}
 \item $|V_i|$ is even for each $i$.\label{enum:even-vertex-sets}
 \item $n\geq 90$ and  $|\{i: |V_i|\, \text{ is odd}\}|>n/2$.\label{enum:large-odd-set}
 \item There are at  least $\Delta(G)-\delta(G)$ disjoint consecutive pairs of vertices.\label{enum:disjoint-consecutive-pairs}
\end{enumerate}
\end{thm}

\begin{theorem}\label{thm:forest}
 If $G$ is a forest then $t(G) \geq \lceil n/2 \rceil - 1$.
\end{theorem}

 If $n$ is odd this is clearly best possible. For even $n$ this is attained, for example, by a star. Such a graph has no perfect twins.

The rest of the paper is organized as follows: In Section~\ref{sec:prelim}, we cite and prove some useful observations about degree sequences and include a short proof of Theorem~\ref{general-bounds} for completness. Sections~\ref{sec:almosttwins}--\ref{sec:forest}
contain the proofs of Theorems ~\ref{thm:almosttwins}--\ref{thm:forest}.

\section{Preliminary and known results}
\label{sec:prelim}

Here, we denote a degree of a vertex $v$ in a graph $G$ by $d(G, v)$ or simply $d(v)$.   For a set $A$ of  vertices,  $d(G,A)=d(A) = \sum_{a\in A} d(a)$.
The set of all $k$-element subsets of a set $X$ is denoted by $\binom{X}{k}$.
For a graph $G$ and a vertex set $W \subseteq V(G)$ the induced subgraph of $G$ on $V(G) \setminus W$ is denoted by $G - W$.

\begin{lemma}\label{degree-pairs}
Let $G$ be an $n$-vertex  graph  and $A, B$ disjoint vertex sets of equal size such that $V(G)=A\cup B$.
If $d(A) = d(B) $ then $e(A)= e(B)$.
If $A$ is a disjoint union of $A_1, \ldots, A_k$ and $B$ is a  disjoint union of $B_1, \ldots, B_k$,  such that $|A_i|=|B_i|$ and $|d(A_i)-d(B_i)|\leq \epsilon$, $i=1, \ldots, k$ then there are disjoint vertex  subsets $A', B'$ of size $n/2$ each such that ${\rm disc}(A',B') \leq \epsilon$.
\end{lemma}

\begin{proof}
 {First assume that $d(A) = d(B)$. Then} we have that $$e(A)-e(B) = (d(A)-e(A,B))/2 - (d(B)-e(A,B))/2 = (d(A)-d(B))/2 = 0.$$
 {Now let $A$ and $B$ be the disjoint union of $A_1,\ldots.A_k$ and $B_1,\ldots,B_k$, respectively, such that $|A_i| = |B_i|$ for each $i = 1,\ldots,k$.} Assume, without loss of generality, that $|d(A_1)-d(B_1)| \geq |d(A_2)-d(B_2)|\geq \cdots \geq |d(A_k)-d(B_k)|$.  Define, for each $i=1,\ldots,k$,
 \begin{equation}\nonumber
  C_i= \begin{cases}
   A_i & i \mbox{ is odd and } d(A_i)\geq d(B_i),\\
   B_i & \mbox{otherwise}.
  \end{cases}
 \end{equation}

 Let $C_i' = \{A_i, B_i\} - C_i$ and consider $A'=C_1 \cup \cdots \cup C_k$ and $B'=C_1'\cup \cdots \cup C_k'$.  Since $e(A')=(d(A')-e(A',B'))/2$ and $e(B')=(d(B')-e(A',B'))/2$, we have $2(e(A')-e(B'))= d(A')-d(B') = \sum_{i=1}^k (d(C_i)-d(C'_i)).$ Since this sum has terms non-increasing in absolute value with alternating signs and bounded in absolute value by $\epsilon$, the sum itself is at most $\epsilon$.
\end{proof}

We shall need the following number theoretic result.

\begin{lemma}\label{fixed-k}
 Let $n\geq 16$  and $k \geq \lg n$. Let $X$ be a multiset  $\{a_1, a_2, \ldots \}$  of  integers from $\{0, 1, \ldots, n-1\}$  such that for any $l\leq k$
 and for any disjoint index sets $\{i_1, \ldots, i_l\}$ and $\{j_1, \ldots, j_l\}$  the following holds:
 $a_{i_1}+a_{i_2}+\cdots+a_{i_l} \neq a_{j_1}+a_{j_2}+\cdots+a_{j_l}$.  Then $|X| < 2k$.
\end{lemma}

\begin{proof}Let $S$ be the set of all possible sums of $k$-element  sub-multi-sets of $X$. Since we assume that all of these are distinct, $|S|=\binom{|X|}{k}$.
On the other hand, there are at most $k(n-1)+1 < kn$ possible sums of $k$ elements from  $X$.
Thus, using a well-known consequence of Stirling's inequality,
$$ kn>|S|=\binom{|X|}{k}\geq\binom{2k}{k}\geq\frac{2^{2k}}{2\sqrt{k}} . $$
Consequently $n>2^{2k-1}/k^{3/2}$. This is a contradiction since $k\geq\lg n \geq 4$.
\end{proof}

Here, we say that two integers are \emph{almost equal} if they differ by $1, 0$, or $-1$.

\begin{thm}[Karolyi~\cite{karolyi}]\label{kar}
Let $X$ be a set of $m$ integers, each between $1$ and $2m-2$.
If $m\geq 89$, then one can partition $X$ into two sets, $X_1$ and $X_2$  of almost equal sizes such that
the sum of elements in $X_1$ is almost equal to the sum of elements in $X_2$.
\end{thm}

For positive integers $k,n$ ($k \leq n/2$), the {\em Kneser graph} $\mathcal{K}= KG(n,k)$ is the graph on the vertex set $\binom{[n]}{k}$ whose edge set consists of pairs of disjoint $k$-sets. Lov\'asz~\cite{lovasz} proved that the chromatic number of the Kneser graph $KG(n,k)$ is $n-2k+2$.
This result is a main tool for proving the lower bound on $t(n)$ used by Caro and Yuster~\cite{caroyuster}. Here we repeat it for completeness.

\begin{proof}[Outline of the  Proof of Theorem~\ref{general-bounds}]
For a graph $G$ with vertex set $[n]$,  we define a vertex coloring of $\mathcal{K}$
by letting the number of edges  induced by each $k$-subset of $V(G)$ be
the color of the corresponding vertex in $\mathcal{K}$.  The number of possible colors is most ${\binom{k}{2}}+1$.
For $k = \sqrt{n}$, we have ${\binom{k}{2}}+1<n-2k+2$. Thus the number of colors is less than the chromatic number of $\mathcal{K}$ and  there are two  {adjacent} vertices in $\mathcal{K}$ of the same color. Therefore there are two disjoint $k$-sets  of $G$ that induce the same number of edges.

To verify the upper bound, consider the disjoint union of cliques of odd orders $a_1, \ldots, a_m$, where $a_j>2(a_1^2+a_2^2+\cdots+a_{j-1}^2)$.
One can show that any pair of twins in this graph must omit at least $c\lg\lg n$ vertices for some constant $c$.
\end{proof}

\section{Proof of Theorem~\ref{thm:almosttwins}}
\label{sec:almosttwins}

Set  $k=\lg n$ and $G=(V,E)$. We shall find two large vertex sets $A$ and $B$ of $G$ of equal size and small discrepancy following the following procedure.\\

\noindent
\textbf{Step 1.} Choose disjoint sets $A_1,B_1\subseteq V$ such that $|A_1|=|B_1|\leq k$ and $d(A_1)=d(B_1)$.\\
\indent By Lemma~\ref{fixed-k} this is possible.

\noindent
\textbf{Step  $i$,  $i>1$.} Choose disjoint sets $A_i,B_i\subseteq V\setminus(A_1\cup B_1\cup\cdots\cup A_{i-1}\cup B_{i-1})$ such that $|A_i|=|B_i|\leq k$ and $d(A_i)=d(B_i)$.\\
\indent Again, by Lemma~\ref{fixed-k} this is possible as long as $|V\setminus (A_1\cup B_1\cup\cdots\cup A_{i-1}\cup B_{i-1})|\geq 2k$.\\

Assume that we have to stop after step $q$ and let $S=V\setminus (A_1\cup\cdots\cup A_q\cup B_1\cup\ldots\cup B_q)$ be the set of leftover vertices. So, $|S|<2k$. Consider the graph $G'=G-S$. Since $d(G,A_i)= d(G,B_i)$ and $V(G)\setminus V(G')=S$, we have that $|d(G',A_i)-d(G',B_i)|=|d(G,A_i)+e(A_i,S)-d(G,B_i)-e(B_i,S)|$. Thus,
$$ |d(G',A_i)-d(G',B_i)|\leq |d(G,A_i)-d(G,B_i)|+|e(A_i,S)-e(B_i,S)|=|e(A_i,S)-e(B_i,S)|\leq |A_i||S|\leq 2k^2.$$

Denoting $\epsilon_i =|d(G',A_i)-d(G',B_i)|$, we have that $\epsilon_i \leq |A_i||S|\leq 2k^2$.  Lemma~\ref{degree-pairs} applied to $G'$ and $A_1,\ldots,A_q,B_1,\ldots,B_q$ asserts the existence of disjoint vertex sets $A$, $B$ in $G$, $|A| = |B| = \sum_{i=1}^q |A_i|$, such that ${\rm disc}(A,B) \leq 2k^2 =2\lg^2 n$.\\

Next, we shall find two vertex sets $A$ and $B$ of $G$ such that $|A|=|B|=\lfloor n/2\rfloor$ and ${\rm disc}(A,B) \leq \frac{1}{2}(\Delta(G) - \delta(G) + 1)$. Let $G' = G$ if $n$ is even, and $G' = G \setminus v$ if $n$ is odd, where $v$ is any vertex of $G$. Thus $G'$ has $2 \lfloor n/2\rfloor$ vertices and $\Delta(G') - \delta(G') \leq \Delta(G) - \delta(G) + 1$. Let $A$, $B$ be two disjoint vertex sets in $G'$ with $|A|=|B|=\lfloor n/2\rfloor$ and minimum discrepancy. It remains to show that ${\rm disc}(A,B) \leq \frac{1}{2}(\Delta(G') - \delta(G'))$.

If $\Delta(G') = \delta(G')$, i.e.  $G'$ is regular,  then ${\rm disc}(A,B) = 0$ by Lemma~\ref{degree-pairs}. Hence we may assume $\Delta(G') - \delta(G') > 0$. Assume, without loss of generality, that ${\rm disc}(A,B) = e(A) - e(B)$. Since $\sum_{a \in A}d(G',a) = 2e(A) + e(A,B) > 2e(B) + e(A,B) = \sum_{b \in B} d(G',b)$ there is a vertex $a \in A$ and a vertex $b \in B$ with $d(G',a) > d(G',b)$. For $A' = (A \setminus a) \cup b$ and $B' = (B \setminus b) \cup a$ we have
$$e(A') - e(B') = e(A) - d(G',a) - e(B) + d(G',b) = {\rm disc}(A,B) - d(G',a) + d(G',b)< e(A)-e(B). $$
Due to the minimality of ${\rm disc}(A,B)$ it follows ${\rm disc}(A',B') = e(B') - e(A')$ and thus
$$ 2{\rm disc}(A,B) \leq {\rm disc}(A',B') + {\rm disc}(A,B) = d(G',a) - d(G',b) \leq \Delta(G') - \delta(G'),$$
which implies ${\rm disc}(A,B) \leq \frac{1}{2}(\Delta(G')-\delta(G'))$ as desired. \hfill $\Box$

\section{Proof of Theorem~\ref{thm:size}}
\label{sec:size}

This argument is very similar to the one used by Caro and Yuster. The only difference is that we are using stronger discrepancy results.
We shall remove a few vertices of large degree. In the remaining graph we find two sets $T$ and $S$,  such that   there are no edges between $T$ and
$S$,   $S$ is large enough, and $T$   induces a matching and independent vertices.  Finding ``almost'' twins guaranteed by Theorem~\ref{thm:almosttwins} in
$G[S]$, and using the edges and vertices of $T$ to balance those, gives us twins in $G$.\\

Let $G$ be an $n$-vertex graph on  $e$ edges.  We shall find large twins $A$ and $B$ of size at least  $\frac{n}{2} \left( 1- \frac{20\sqrt{e} \lg n}{n}\right)$ each.
To shorten the expressions we use $f = f(n,e) =  \frac{\sqrt{e} \lg n}{n}$, so we are looking for twins of size at least $\frac{n}{2} (1-20f)$ each. In particular we may assume that $n(1-20f) > 0$.
Assume that the size of the largest independent set in $G$ is less than $n(1- 20f)$, otherwise we find the desired twins as subsets of this independent set.
Assume also that  $e\geq 4$. \\

Consider the set $L= L(x)$ of vertices of degree at least $x= \frac{2e}{nf}$.
The number of edges incident to these vertices is at least $ |L|x/ 2$.
Since this number is at most the total number $e$ of edges,  $|L|\leq  \frac{2e}{x} = nf$.
Let $G'=G-L$, a graph with maximum degree less than $x$ and number of vertices at least $n(1-f)$.
Next we shall consider $G'$ only.

Let $l = 2\frac{(nf)^2}{e} = 2\lg^2 n$. For a vertex $v$, denote by $N[v]$ the set of vertices consisting of $v$ together with its neighbors in $G'$. For an edge $e_0=uv$, denote by $N[e_0]$ the set of vertices $\{u,v\}$ together with each of their neighbors in $G'$. We shall choose a set $T$ of $4l$ vertices from $V=V(G')$ inducing a matching $e_1, \ldots, e_l$ and $2l$ independent vertices $v_1, \ldots, v_{2l}$
 via a simple greedy procedure:\\

 \noindent
{\bf Step 1.}  Pick an edge $e_1$ with endpoints in $V$.\\
  \indent Since $|V| \geq n(1 - f)$, $V$ can not be an independent set, so the edge $e_1$ exists.\\
{\bf Step $i$, $1<i\leq l$.}  Pick an edge $e_i$ with endpoints in $V \setminus (N[e_1] \cup \cdots \cup  N[e_{i-1}])$.\\
  \indent  Since the size of  $V \setminus (N[e_1] \cup \cdots \cup N[e_{i-1}])$ is greater than $n(1-f) - 2(i-1)x > n(1-f) - 2lx = n(1-f) - 2\cdot 2\frac{(nf)^2}{e}\frac{2e}{nf} = n(1-9f) > 0$, this set is not an independent set and an edge $e_i$ can be selected for $i=1,\ldots,l$.\\

\noindent
{\bf Step $l+1$.}  Pick a vertex $v_1$ from $V \setminus (N[e_1] \cup \cdots \cup N[e_l])$.\\
{\bf Step $l+j$, $1<j\leq 2l$.}  Pick a vertex $v_j$ from $V \setminus (N[e_1] \cup \cdots \cup N[e_l] \cup N[v_1] \cup \cdots \cup N[v_{j-1}])$.\\
  \indent The fact that  this set is non-empty for each $j$ follows from the argumentation below.

Let $S$ be the last set $S= V \setminus (N[e_1] \cup \cdots \cup N[e_l] \cup N[v_1] \cup \cdots \cup N[v_{2l}])$.  Then $|S| >  n(1-f) - 4lx = n(1-17f) > 0$. Moreover no vertex of $S$ is adjacent to any vertex of $T$. Applying Theorem~\ref{thm:almosttwins} to $G[S]$ gives two disjoint subsets of $S$, $A$ and $B$, of the same size at  least $(|S|-2\lg |S|)/2$ and with discrepancy $\gamma$, $\gamma \leq 2\lg^2(|S|) \leq 2\lg^2 n \leq l$. Assume, without loss of generality, that $e(A)\geq e(B)$. We construct a set $A'$ by adding $2\gamma \leq 2l$ independent vertices from $T$ to $A$, and a set $B'$ by adding $\gamma \leq l$ edges from $T$ to $B$. Now, the discrepancy of $A'$ and $B'$ in $G$ is zero since there are no edges between $A\cup B$ and $T$. Thus $A'$ and $B'$ are twins in $G$ of size at least $(|S| - 2\lg |S|)/2 > \frac{n}{2}(1-17f) - \lg n = \frac{n}{2}(1-17f) - \frac{nf}{\sqrt{e}} \geq \frac{n}{2}(1-19f)$. \hfill $\Box$

\section{Proof of Theorem~\ref{thm:exact}}
\label{sec:exact}

 {Let $G$ be a graph on $n$ vertices, where $n$ is even. We shall show that if one of conditions~\ref{enum:consecutive-degrees}-\ref{enum:disjoint-consecutive-pairs} holds then $t(G)=n/2$, i.e., the vertex set of $G$ can be perfectly split into twins.}

 \begin{enumerate}
  \item[\ref{enum:consecutive-degrees}.]  {Assume that the degrees of $G$ form a set of consecutive integers. That is, for every integer $i$, $\delta(G)\leq i\leq\Delta(G)$, there exists a vertex of degree $i$.} Let $A$ and $B$ be sets that have equal cardinality that partition the vertex set and have the smallest discrepancy. If the discrepancy is nonzero, assume that $e(A)<e(B)$. In particular, we see that if $A$ has a vertex of degree $d$ then $B$ does not have a vertex of degree $d+1$ (otherwise swapping the two will give a partition with smaller discrepancy). Similarly, if $B$ has a vertex of degree $d$, then $A$ has no vertex of degree $d-1$. So, if $d$ is the smallest degree of a vertex in $A$ then $A$ contains all vertices of degrees $d+1, d+2, \ldots, \Delta(G)$. Similarly, if $d'$ is the highest degree of a vertex of $B$, then $B$ contains all vertices of degree $d'-1, d'-2, \ldots, \delta(G)$. Moreover $d'\leq d$. So, we have that the sum of the degrees in $A$ is strictly greater than the sum of
the degrees in $B$, a contradiction to the assumption that $e(A)<e(B)$.

  \item[\ref{enum:even-vertex-sets}.]  {Assume that the number of vertices of degree $i$ is even for each integer $i$, $\delta(G)\leq i\leq\Delta(G)$.} Observe using Lemma~\ref{degree-pairs} that it is sufficient to partition the vertex set into two parts of equal sizes such that the degrees of vertices in each part have the same sum. Since each degree appears an even number of times, we put half of the corresponding vertices in one part and half in another.

  \item[\ref{enum:large-odd-set}.]  {Assume that $n\geq 90$ and $|\{i: |V_i|\, \text{ is odd}\}|>n/2$.} For each degree $d$, put $\lfloor d/2 \rfloor$ vertices of degree $d$ in one part and $\lfloor d/2 \rfloor$ vertices of degree $d$ in another part. Let $S$ be the set of remaining vertices of distinct degrees. We shall split $S$ also using Theorem~\ref{kar}. Formally, for each degree $d$, let $V_d = A_d \cup B_d \cup S_d$, where $|A_d|=|B_d|$ and $|S_d| \leq 1$. Then $S = S_{\delta}\cup \cdots \cup S_{\Delta}$ is an even set of more than $n/2$ vertices of distinct degrees. Applying Theorem~\ref{kar} to the degrees of vertices in $S$, we can split $S$ in two parts, $U$ and $U'$ of equal sizes such that $|d(U)-d(U')|\leq 1$. Let $A = U\cup  \bigcup_{d=\delta}^{\Delta} A_d$ and $B = U'\cup \bigcup_{d=\delta}^{\Delta} B_d$. Since $d(V)$ is even and $d\left(\bigcup_{d=\delta}^{\Delta} A_d\right) = d\left(\bigcup_{d=\delta}^{\Delta} B_d\right)$, then $d(U)$ and $d(U')$ must have the same parity, and therefore
must be equal. So it must be the case that ${\rm disc}(A, B) =0$.

  \item[\ref{enum:disjoint-consecutive-pairs}.] {Assume that there are least $\Delta-\delta$ disjoint pairs of vertices whose degrees differ by exactly $1$.} With $x=\Delta-\delta$, there is a set of vertices $\{v_1,v_1',\ldots,v_x,v'_x\}$ such that $d(v_i)=d(v'_i)-1$. Partition $V(G) - \{v_1,v'_1,\ldots,v_x,v'_x\}$ into two sets $A$ and $B$, of equal cardinality with $0\leq d(A) - d(B) = x'$, $x'\leq x$.  Note that this could be done greedily
  by ordering the vertices in the non-increasing order of their degrees and alternately adding one vertex to $A$ and the next to $B$.
  Let $A'=A \cup \{v_1,\ldots,v_{x'}\}$ and $B'=B \cup \{v'_1,\ldots,v'_{x'}\}$. As a result, $d(A') = d(B')$. If $x'\neq x$, let $A''= A'\cup \{v_i: x\geq i>x', i \mbox{  is  even}\} \cup \{v'_i: x\geq i>x', i \mbox{  is  odd}\}$ and $B''= B' \cup \{v'_i: x\geq i>x', i \mbox{  is  even}\} \cup \{v_i: x\geq i>x', i \mbox{  is  odd}\}$. Consequently, $|d(A'') -d(B'')|\leq 1$. However, since the sum of the degrees of a graph is even and $V(G)= A''\cup B''$, we have that $d(A'')=d(B'')$. Therefore $A''$ and $B''$ are twins of size $n/2$ each. We note that this technique is
reminiscent of the one used by Ben-Eliezer and Krivelevich in~\cite{BK}.
 \end{enumerate}~ \hfill $\Box$

\section{Proof of Theorem~\ref{thm:forest}}
\label{sec:forest}

First, we observe that we may assume that no component of our forest consists of any isolated edges.  If it does, remove those edges and find the twins in the remaining graph and delete $0$, $1$ or $2$ vertices to get twins $A$, $B$.  We can add the endvertices of each isolated edge the twins -- one to $A$ and one to $B$  -- to obtain twins for the whole forest, apart from the at most $2$ vertices that were already deleted.

In order to proceed, we need to define a type of twins with additional structure. Let $G=(V, E)$ be a forest and $A$, $B$ be twins.  We say that $A\cup B$ is the set of \emph{colored} vertices and $G[A\cup B]$ is the \emph{colored graph}; the remaining vertices are \emph{uncolored}.
We define $A$ and $B$ to be \textbf{good twins} if
\begin{itemize}
 \item{}  each uncolored vertex has degree at most $1$ in $G$,
 \item{}  the set of colored vertices induces no isolated vertices,
 \item{}  all colored vertices except for a set $S\subseteq A$ of at  most two vertices  have all neighbors colored,
 \item{}  $S$ contains at most one leaf and at most one non-leaf of the colored graph,
 \item{}  if $S$ contains a non-leaf, $w$,  of the colored graph, then $w$ has at most one neighbor in $A$,
 \item{}  if  $S$ contains a non-leaf $w$ and  a leaf $v$ of the colored graph,     then $v$ has smaller number of uncolored neighbors than $w$.
\end{itemize}

Because each uncolored vertex has degree at most $1$ and there are no isolated edges, if $A$,$B$ are good twins then at least one vertex of each edge is colored. Therefore, the set of edges that have exactly one colored vertex come in at most two disjoint stars, each centered at a vertex in $S\subseteq A$. We refer to the middle and the right of Figure~\ref{fig:tree-example} for examples of good twins. The figure exemplifies one step in the inductive proof of existence of good twins (Theorem~\ref{thm:good_twins}).

\begin{figure}[htbp]
 \includegraphics{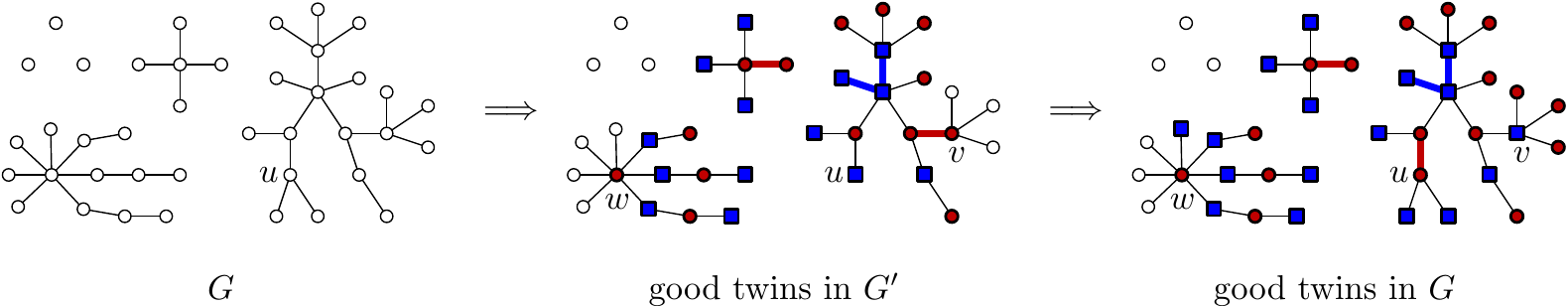}
 \caption{Illustration of the proof of Theorem~\ref{thm:good_twins}. Left: A forest $G$ and a vertex $u$ adjacent to at least one leaf and at most one non-leaf. Middle: Good twins in the graph $G'$ obtained from $G$ by deleting all leaves adjacent to $u$. Right: Good twins in $G$ obtained by recoloring $u$ and $v$ and applying a $(u,v)$-move followed by a $(w,v)$-move -- c.f. Case~3.1.b. We use red circles to indicate members of $A$ and blue squares to indicate members of $B$.}
 \label{fig:tree-example}
\end{figure}

\begin{theorem}\label{thm:good_twins}
In every forest $G$ there exist good twins.
\end{theorem}
\begin{proof}

Let $G = (V,E)$ be a forest. We do induction on the number of edges of $G$. If $|E| = 0$, then $G$ consists of isolated vertices only and good twins of $G$ are given by $A = B = \emptyset$.

For the induction step let $u$ be any vertex in $G$ that is adjacent to at least one leaf of $G$ and to at most one non-leaf of $G$. We apply induction to the subgraph $G'$ of $G$ obtained by deleting all leaves adjacent to $u$ and obtain good twins $A$, $B$ of $G'$. Note that $u$ is either a leaf or an isolated vertex of $G'$.

What follows is a case analysis. In each case we will argue how to modify the good twins $A$, $B$ of $G'$ so as to obtain good twins of $G$. The modification of twins in each case is done in one or two steps called \textit{moves} that are defined for a pair of vertices.

Let $X$ and $Y$ be two sets of vertices, $x\in X$, $y\in Y$, for any vertex $v$, $L(v) = N(v) \setminus (X \cup Y)$. If $L(x)$ and $L(y)$ are disjoint, then an \textbf{$(x,y)$-move with respect to $X,Y$} (or an \textbf{$(x,y)$-move}, where the sets $X$ and $Y$ are understood) is an operation creating a pair of sets $X',Y'$ as follows:

\begin{itemize}
  \item Let $L'(x) \subseteq L(x)$ and $L'(y) \subseteq L(y)$ be sets with $|L'(x)| = |L'(y)| = \min\{|L(x)|,|L(y)|\}$.
  \item Let $X' = X\cup L'(y)$ and $Y' = Y\cup L'(x)$.
\end{itemize}

Note that if $X$ and $Y$ are twins and $L(x)$ and $L(y)$ contain only leaves of $G$, the $(x,y)$-move creates twins again. Moreover, all neighbors of at least one of $x,y$ are contained in the new twins $X' \cup Y'$. See Figure~\ref{fig:xy-move} for an illustration. We shall see an $(x,y)$-move as a map $f$ that takes two sets of vertices $X$, $Y$ and two vertices $x \in X$, $y \in Y$ and maps these to $X'$, $Y'$, defined as above, i.e., $f(x,y,X,Y) = (X',Y')$.

\begin{figure}[htbp]
 \includegraphics{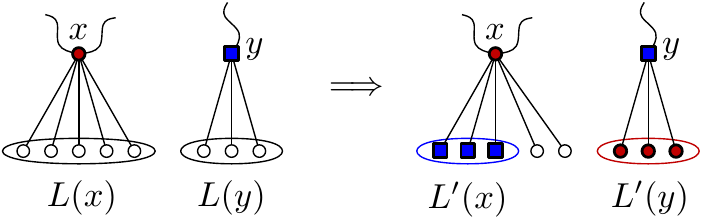}
 \caption{An $(x,y)$-move with respect to the sets $X$ of red vertices and the set $Y$ of blue vertices.}
 \label{fig:xy-move}
\end{figure}


Let $S \subseteq A$, $|S| \leq 2$, be the set of those vertices in $G'$ that have uncolored neighbors. Then in $G$ the set of non-leaf vertices with uncolored neighbors is given by $S \cup u$. We distinguish the cases $|S|=0$ (Case~1), $|S|=1$ (Case~2) and $|S|=2$ (Case~3), and the subcases $u \in A$, $u \in B$ and $u \notin A \cup B$. We apply $(x,y)$-moves to twins only, with $\{x,y\} \subseteq S \cup u$, each time decreasing the number of vertices that have uncolored neighbors by at least one. In most of the cases we end up with at most one vertex $x$ with uncolored neighbors, and to see that these twins are good twins it then suffices to check that $x$ is incident to at most one edge in $G[A]$ or $G[B]$. To be precise, we may have to swap the roles of $A$ and $B$ in the case where $x \in B$. However, due to simplicity of presentation we omit to mention this explicitly.

In many cases we do not apply $(x,y)$-moves immediately. Instead, we first color some uncolored vertices and/or recolor some leaves of the colored graph. In the first case, we always add to $A$ some vertex that has no neighbor in $A$ and to $B$ some  vertex that has no neighbor in $B$, i.e., we get twins again. If we recolor a leaf $x$, say $x \in A$, in the colored graph this has the following effects: $|A \setminus x| = |B \cup x| - 2$ and $e(A \setminus x) = e(B \cup x) - 1$. We repair this imbalance by either recoloring a leaf $y\in B$ in the colored graph or adding two uncolored vertices to $A$, one with no neighbor in $A$ and the other with exactly one neighbor in $A$. See Figure~\ref{fig:case2_4} and~\ref{fig:case2_5} for an example.

 \begin{itemize}
  \item \textit{Case 1 -- $|S|=0$:} If $u \notin A \cup B$, then $u$ is an isolated vertex of $G'$ and good twins of $G$ are given by $A \cup \{u\}$, $B \cup \{x\}$ for some vertex $x \in L(u)$. If $u \in A \cup B$ then $A$, $B$ are already good twins for $G$. In both cases, the new set $S$ is $\{u\}$.

\medskip

 \begin{figure}[htbp]
   \begin{center}
    \subfigure[Case~2.1 -- $u \in B$:\newline $(A',B') = f(w,u,A,B)$]{
     \includegraphics{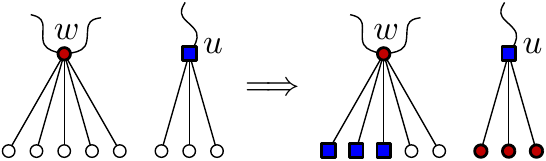}
     \label{fig:case2_1}
    }
    \\
    \hspace{2em}
    \subfigure[Case~2.2.a -- $u \in L(w)$:\newline $(A',B') = f(w,u,A \cup x,B \cup u)$]{
     \includegraphics{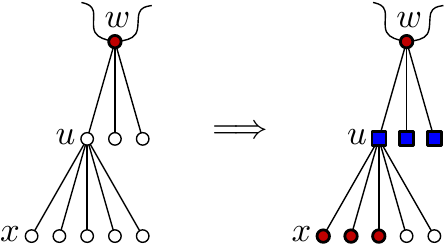}
     \label{fig:case2_2}
    }
    \hspace{2em}
    \subfigure[Case~2.2.b -- $u$ has no colored neighbor:\newline $(A',B') = f(w,u,A \cup x, B \cup u)$]{
     \includegraphics{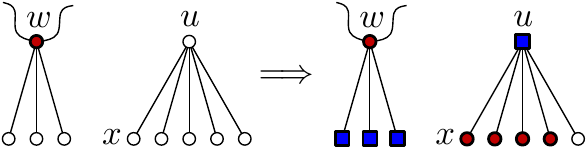}
     \label{fig:case2_3}
    }
    \hspace{2em}
    \subfigure[Case~2.3.a -- $u \in A$, $|L(w)| \leq |L(u)|$ and its colored neighbor is in $A$:\newline $(A',B')=f(w,u,(A \setminus u)\cup \{x,y\}, B\cup u)$]{
     \includegraphics{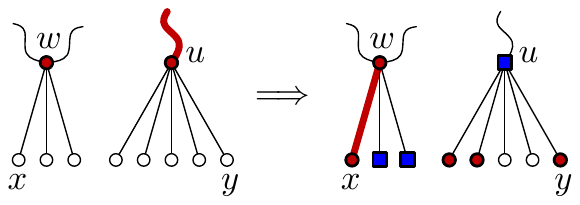}
     \label{fig:case2_4}
    }
    \hspace{2em}
    \subfigure[Case~2.3.b -- $u \in A$, $|L(w)| \leq |L(u)|$ and its colored neighbor is in $B$:\newline $(A',B')=f(w,u,(A \setminus u) \cup \{x,y\}, B\cup u)$]{
     \includegraphics{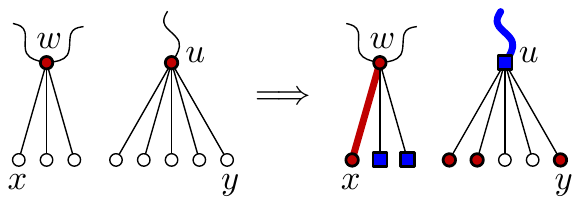}
     \label{fig:case2_5}
    }
   \end{center}
   \caption{Illustration of Case~2 in the proof of Theorem~\ref{thm:good_twins}.}
   \label{fig:case2}
  \end{figure}

  \item \textit{Case 2 -- $|S|=1$:} Let $S = \{w\} \subseteq A$. We may assume, without loss of generality, that $w$ is a non-leaf of $G[A\cup B]$. Next, we distinguish several cases depending on whether $u \in A$, $u \in B$, or $u \notin A \cup B$. Figure~\ref{fig:case2} shows for each case how to obtain good twins for $G$ by (re)coloring some vertices and applying a $(w,u)$-move. The procedure is formally described as follows: \smallskip

      \begin{itemize}
        \item \textit{Case~2.1-- $u \in B$:} We perform a single $(w,u)$-move. \smallskip
        \item \textit{Case~2.2-- $u\not\in A\cup B$:} Either $u \in L(w)$ or $u$ was isolated in $G'$. In either case, we first add $u$ to $B$ and any $x \in L(u)$ to $A$. We obtain twins again, to which we perform a $(w,u)$-move. If $|L(w)|>|L(u)|$, then the set $S$ remains $\{w\}$. If $|L(w)|<|L(u)|$, then the roles of $A$ and $B$ are switched and $S$ becomes $\{u\}$. If $|L(w)|=|L(u)|$, then $S=\emptyset$. \smallskip
        \item \textit{Case~2.3 -- $u \in A$:} If $w$ is a non-leaf of $G'[A\cup B]$ and $|L(w)| > |L(u)|$, then $A$, $B$ are already good twins of $G$ and the new set $S$ is $\{w,u\}$. In fact, we may assume that $|L(w)| \leq |L(u)|$, since if $w$ is a leaf of $G'[A\cup B]$ and $|L(w)| > |L(u)|$, we can swap the roles of $w$ and $u$. First, we recolor $u$ (remove it from $A$ and add it to $B$), add some $x \in L(w)$ as well as some $y \in L(u)$ to $A$, and perform a $(w,u)$-move to the resulting twins. Figure~\ref{fig:case2_4} and~\ref{fig:case2_5} depicts the case that the colored neighbor of $u$ is in $A$ and $B$, respectively.\smallskip
      \end{itemize}
In Case~3 below when we recolor a vertex $x$ we do not distinguish the subcases whether the colored neighbor of $x$ is in $A$ or $B$, since the treatment is the same. We indicate in Figure~\ref{fig:case3} the first subcase only.

\medskip
 \begin{figure}[htbp]
   \begin{center}
    \subfigure[Case~3.1.a -- $u \in B$ and $|L(u)| \geq |L(v)|$:\newline $(A',B') = f(w,u,f(v,u,A,B))$]{
     \includegraphics{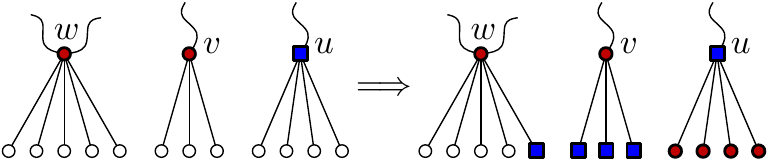}
     \label{fig:case3_1A}
    }
    \hspace{1em}
    \subfigure[Case~3.1.b -- $u \in B$ and $|L(u)| < |L(v)|$:\newline $(A',B') = f(w,v,f(u,v,(A \setminus v) \cup u, (B \setminus u) \cup v))$]{
     \includegraphics{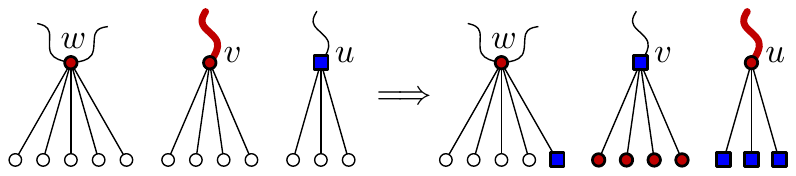}
     \label{fig:case3_1B}
    }
    \hspace{2em}
    \subfigure[Case~3.2.a -- $u \in L(w)$, $|L(w)| \leq |L(u)| + |L(v)|$:\newline $(A',B') = f(w,v,f(w,u,(A \setminus v) \cup \{x,y,z\}, B \cup \{u,v\}))$\newline(The same is done if $u$ has no colored neighbor.)]{
     \includegraphics{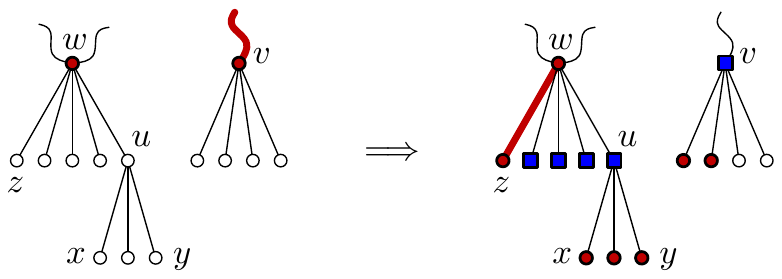}
     \label{fig:case3_2A}
    }
    \hspace{2em}
    \subfigure[Case~3.2.b -- $u \in L(w)$, $|L(w)| > |L(u)| + |L(v)|$:\newline $(A',B') = f(w,u,A \cup x, B \cup u)$\newline(The same is done if $u$ has no colored neighbor.)]{
     \includegraphics{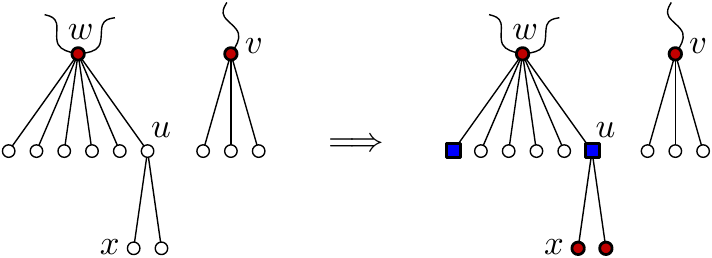}
     \label{fig:case3_2B}
    }
    \hspace{2em}
    \subfigure[Case~3.3.a -- $u \in L(v)$, $|L(u)| \geq |L(v)|$:\newline $(A',B') = f(w,u,f(v,u,A \cup x, B \cup u))$]{
     \includegraphics{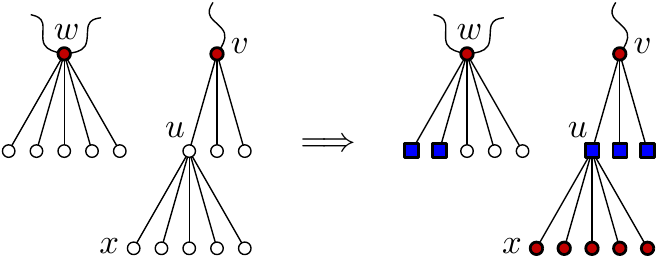}
     \label{fig:case3_3A}
    }
    \hspace{2em}
    \subfigure[Case~3.3.b -- $u \in L(v)$, $|L(u)| < |L(v)|$:\newline $(A',B') = f(w,v,f(u,v,(A \setminus v) \cup \{u,x\},B \cup v))$]{
     \includegraphics{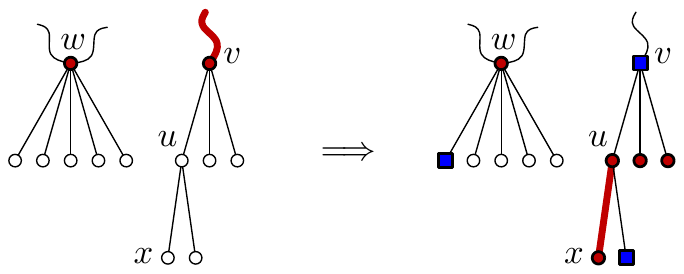}
     \label{fig:case3_3B}
    }
    \hspace{2em}
    \subfigure[Case~3.4 -- $u \in A$, $|L(u)| \geq |L(v)|$:\newline $(A',B') = f(w,u,f(v,u,(A \setminus u) \cup \{x,y\},B \cup u))$]{
     \includegraphics{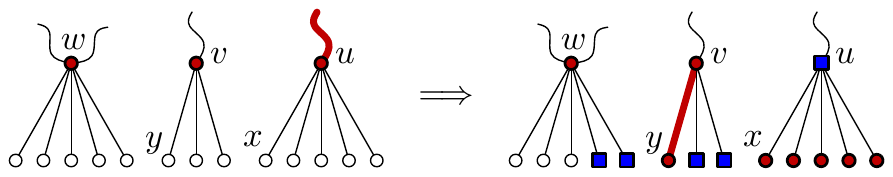}
     \label{fig:case3_4}
    }
   \end{center}
   \caption{Illustration of Case~3 in the proof of Theorem~\ref{thm:good_twins}.}
   \label{fig:case3}
  \end{figure}

  \item \textit{Case 3 -- $|S|=2$:} By the definition of good twins we have $S = \{v,w\} \subseteq A$, $v$ a leaf and $w$ a non-leaf of $G'[A\cup B]$, and $|L(v)| < |L(w)|$. We distinguish, respectively, whether $u\in B$, $u\in L(w)$, $u\in L(v)$ or $u\in A$. We further distinguish subcases depending on the sizes of $L(u)$, $L(v)$ and $L(w)$. The treatment for each case is shown in Figure~\ref{fig:case3}; the description follows. \smallskip
            \begin{itemize}
      \item \textit{Case~3.1.a-- $u \in B$ and $|L(u)| \geq |L(v)|$:} We perform a $(v,u)$-move and afterwards a $(w,u)$-move. Note that if $|L(u)|=|L(v)|$, i.e., $u$ has no uncolored neighbors after the $(v,u)$-move, the second move has no effect. \smallskip
      \item \textit{Case~3.1.b-- $u \in B$ and $|L(u)| < |L(v)|$:} In particular, this means $|L(u)| < |L(w)|$. Then we swap the colors of $u$ and $v$, i.e., remove $u$ from $B$ and add it to $A$ as well as removing $v$ from $A$ and adding it to $B$. Since both $u$ and $v$ are leaves in the colored graph, we obtain twins again. Then apply a $(u,v)$-move followed by a $(w,v)$-move. \smallskip
      \item \textit{Case~3.2.a-- either $u \in L(w)$ or $u$ was  isolated in $G'$ and $|L(w)| \leq |L(u)|+|L(v)|$:} We add $u$ to $B$ and some $x \in L(u)$ to $A$, obtaining twins again. If $|L(w)|<|L(u)|$, then perform a $(w,u)$-move followed by a $(v,u)$-move. If $|L(w)|\geq |L(u)|$, then, recolor $v$ and add some $y \in L(u)$ (if $|L(u)| = 1$ take some $y \in L(v)$ instead) and $z \in L(w)$ to $A$. To the resulting twins we apply a $(w,u)$-move followed by a $(w,v)$-move. \smallskip
      \item \textit{Case~3.2.b-- either $u \in L(w)$ or $u$ was uncolored and isolated in $G'$ and $|L(w)|>|L(u)|+|L(v)|$:} We again add $u$ to $B$ and some $x \in L(u)$ to $A$. To the resulting twins apply a $(w,u)$-move only. Note that in this case, both $v$ and $w$ still have uncolored neighbors. \smallskip
      \item \textit{Case~3.3.a-- $u \in L(v)$ and $|L(v)|\leq |L(u)|$:} We add $u$ to $B$ and some $x \in L(u)$ to $A$, and then perform a $(v,u)$-move followed by a $(w,u)$-move. \smallskip
      \item \textit{Case~3.3.b-- $u \in L(v)$ and $|L(v)| > |L(u)|$:} We recolor $v$ and compensate the imbalanced edge-count and vertex-count by adding $u$ as well as some $x\in L(u)$ to $A$. We perform a $(u,v)$-move followed by a $(w,v)$-move. \smallskip
      \item \textit{Case~3.4-- $u \in A$:}  Assume, without loss of generality, that $|L(u)| \geq |L(v)|$ (otherwise swap the roles of $u$ and $v$). We recolor $u$ and add some $x \in L(u)$ as well as some $y \in L(v)$ to $A$. This produces twins to which we perform a $(v,u)$-move and afterwards a $(w,u)$-move. \smallskip
      \end{itemize}
 \end{itemize}
 It is straightforward to check that our treatment in each of the above cases results in good twins $A',B'$ of $G$. The formal definition of $(A',B')$ in Case~2 and Case~3 is given in Figure~\ref{fig:case2} and Figure~\ref{fig:case3}, respectively.
\end{proof}

Having Theorem~\ref{thm:good_twins}, we can prove Theorem~\ref{thm:forest} quite easily.

\begin{proof}[Proof of Theorem~\ref{thm:forest}]
 Let $G = (V,E)$ be a forest. By Theorem~\ref{thm:good_twins}, $G$ admits good twins $A$, $B$. The vertices in $V \setminus (A\cup B)$, i.e., the uncolored vertices, come in three groups, each of which may be empty. The first group $S_1$ contains all isolated vertices of $G$. The second group $S_2$ consists of some leaves of $G$ that are all adjacent to a leaf $v\in A$ of $G[A\cup B]$. The third group $S_3$ consists of some leaves of $G$ that are all adjacent to a non-leaf $w \in A$ of $G[A\cup B]$. Moreover, if $S_3 \neq \emptyset$, then $|S_3| > |S_2|$.\smallskip

\begin{itemize}
  \item \textit{Case 1 -- $S_2 = S_3 = \emptyset$:} Partition $S_1$ into equal-sized subsets by removing at most one vertex, and add one subset to each of $A$, $B$. \smallskip

  \item \textit{Case 2 -- exactly one of $S_2$, $S_3$ is non-empty:}   We will delete the corresponding vertex $v$ or $w$ from $A$, say it's $w$. I.e., we will uncolor $w$ but color some of its neighbors.

  If $w$ has no neighbor in $A$, then uncolor $w$ and add all but at most one of the remaining uncolored vertices to $A$ and to $B$ evenly so that the resulting sets are twins of size $\lceil n/2\rceil-1$. So, we may assume that $w$ has a neighbor in $A$ (exactly one, by the condition of twins being good).

  It is easy to see that $G$ has a leaf in $B$. For example, contracting all edges in $G[A]$ and in $G[B]$ and removing the uncolored vertices in $S$ gives a forest $\tilde{G}$ whose bipartition classes are induced by $A$ and $B$ and are of equal size. Every leaf of $\tilde{G}$ in $B$ corresponds to at least one leaf of $G$ in $B$ and every component of $\tilde{G}$ with at least as many vertices in $B$ as in $A$ has a leaf in $B$.
  So let $u \in B$ be a leaf of $G$. Then uncolor $w$ and recolor $u$ to be in $A$. Add all but at most one of the remaining uncolored vertices to $A$ and to $B$ evenly so that the resulting sets are twins of size $\lceil n/2\rceil-1$.

   \item \textit{Case 3 -- $S_2 \neq \emptyset$ and $S_3 \neq \emptyset$:} If $w$ has a neighbor in $A$, then add some vertex $x \in S_2$ to $A$. If $w$ has no neighbor in $A$, then and some vertex $y \in S_3$ to $A$. Uncolor $w$. These are twins again.

       Next, add all vertices in $S_2$, except possibly $x$, to $B$, and add the same number of vertices from $S_3$ to $A$. (Note that $|S_3| > |S_2|$ and hence $|S_3 \setminus y| \geq |S_2|$.) In the resulting twins, no uncolored vertex has a colored neighbor and thus we can find twins of $G$ by removing at most one further vertex. In any case, the graph $G$ contains at most 2 uncolored vertices and the colors form twins.
  \end{itemize}
This completes the proof.
\end{proof}

\bibliographystyle{plain}

\end{document}